\documentclass{article}
\usepackage{graphicx}

\usepackage[fleqn]{amsmath}

\usepackage{amsthm}
\usepackage{amssymb}
\usepackage{amsbsy}
\usepackage{amsmath}
\usepackage{amsfonts}
\usepackage{mathrsfs}
\usepackage{amsmath}
\usepackage[all]{xy}
\usepackage{amstext}
\usepackage{amscd}
\usepackage[dvips]{epsfig,color}
\usepackage{psfrag}
\usepackage{enumerate}
\usepackage{flafter}
\allowdisplaybreaks

\theoremstyle{plain}
\newtheorem{thm}{Theorem}[section]
\newtheorem{prop}[thm]{Proposition}
\newtheorem{cor}[thm]{Corollary}

\theoremstyle{definition}

\newtheorem{conj}[thm]{Conjecture}

\newtheorem{prob}[thm]{Problem}

\def\4{\mathop{4 \mathrm{f}}\nolimits}

\def\12{\mathop{X_{12}}\nolimits}

\def\Hom{\mathop{\mathrm{Hom}}\nolimits}

\newcommand{\lra}{\longrightarrow}
\newcommand{\ra}{\rightarrow}
\newcommand{\Q}{{\Bbb Q}}
\newcommand{\R}{{\Bbb R}}
\newcommand{\Z}{{\Bbb Z}}

\newcommand{\C}{{\Bbb C}}

\newcommand{\SL}{{\rm SL}}
\newcommand{\CS}{{\rm CS}}
\newcommand{\Li}{{\rm Li}}
\newcommand{\Log}{{\rm Log}}

\begin{document} 

\large 

\begin{center} 
{\bf \Large Reciprocity of the Chern-Simons invariants of 3-manifolds} 
\end{center} 
\begin{center}{Takefumi Nosaka 
}\end{center} 
\begin{abstract}\baselineskip=12pt \noindent 
We pose reciprocity conjectures of the Chern-Simons invariants of 3-manifolds, and discuss some supporting evidence on the conjectures. Especially, we show that the conjectures hold if Galois descent of a certain $K_3$-group is satisfied. 
\end{abstract} 
\begin{center} 

\normalsize 

\baselineskip=11pt 

{\bf Keywords} \\ 

3-manifolds, Chern-Simons invariants, dilogarithm identities,  

\end{center} 

\begin{center} 

\normalsize 

\baselineskip=11pt 

{\bf Subject Codes } \\ 

\ \ \ \ \ \ 57K31, 58J28, 11G55, 19J10, 11G55 \ \ \ 

\end{center}

\large 

\baselineskip=16pt 

\section{Introduction} 

Let us review the Chern-Simons invariant. Let $M $ be an oriented closed 3-manifold, and $G$ be a complex simple Lie group, which is simply connected. Let $\mathrm{ R}_G^{\rm irr}(M) $ be the character variety of $M$, that is, the set of the conjugacy classes of all irreducible representations $ \rho:\pi_1(M) \ra G.$ For $[\rho] \in \mathrm{ R}^{\rm irr}_G(M) $, consider the associated principal $G$-bundle $P \ra M$ with a flat connection $A_{\rho}$, and choose a smooth section $\mathfrak{s}: M\ra P$. Then, as in \cite{CS}, with a choice of a (normalized) Killing form $P : \mathfrak{g}^2 \ra \C$, the Chern-Simons invariant of $\rho$ is defined to be 
\begin{equation}\label{sum22} \CS(\rho)= \int_{M} \mathfrak{s}^*( P ( A_{\rho} \wedge [A_{\rho}, A_{\rho}] )) \in \C/ \Z, \end{equation} 
modulo $\Z$. For example, if $ G= \SL_n(\C)$, then $ P ( A_{\rho} \wedge [A_{\rho}, A_{\rho}] ) $ is the familiar formula $\mathrm{Tr}(A_{\rho} \wedge {\rm d} A_{\rho} + \frac{2}{3 } A_{\rho} \wedge A_{\rho} \wedge A_{\rho} ) / 8\pi^2 $. Since the map $\CS: \mathrm{ R}_G^{\rm irr}(M) \ra \C/\Z $ is known to be locally constant, it can be regarded as a map from $\pi_0(\mathrm{ R}_G^{\rm irr}(M) )$, where $\pi_0(\mathrm{ R}_G^{\rm irr}(M) )$ is the connected components of $\mathrm{ R}_G^{\rm irr}(M) $ and is of finite order. However, in many papers, the Chern-Simons invariant is studied in situations with a fixed $\rho$. For example, if $M$ is hyperbolic and $G=\SL_2(\C)$, there is an associated holonomy $\rho_M : \pi_1(M)\ra G$ such that the quadrupled invariant $4 \CS(\rho_M) $ equals the complex volume of $M$, with relation to the algebraic $K$-theory, Beilinson regulator, $\eta$-invariant, etc.

Meanwhile, motivated by physical predictions from M5-branes in string theory, the sums of some topological invariants with respect to all $\rho \in \pi_0(\mathrm{ R}_G^{\rm irr }(M) ) $ have recently been discussed. For example, it is conjectured that the sum of adjoint Reidemeister torsions (or, generally, some perturbative invariants of odd order) vanishes; see \cite{BGZ, CGK, CQW,Yoon}. Analogously, in this paper, we focus on the sum of the Chern-Simons invariants and discuss the vanishing of the sum. Precisely, the main problem can be described as follows: 
\begin{prob}\label{conj12} 
Is there an integer $c_{G} \in \Z $ such that the vanishing of the following sum holds for any closed 3-manifold $M$? 
\begin{equation}\label{sum} c_G \sum_{[\rho] \in \pi_0( \mathrm{ R}^{\rm irr}_G(M)) } \CS (\rho) \in \mathbb{C}/\Z \end{equation} 
If it does vanish, provide a topological interpretation of $r_{G,M} \in \Z/c_G\Z$ such that  
\begin{equation}\label{sum3 } \sum_{[\rho] \in \pi_0( \mathrm{ R}^{\rm irr}_{G }(M)) } \CS (\rho) = \frac{r_{G,M}}{c_G} \in \C/\Z. \end{equation} 
\end{prob} 

\noindent 
If this problem is multi-directionally proved/improved in terms of, e.g., $K$-theory, differential geometry, or mathematical physics, it might further our understanding of the intersections of the above sums.

In this paper, we restrict ourselves to the case $G=\SL_2(\C) $ and establish a conjecture. 
\begin{conj}\label{conj1} 
If $ G= \SL_2(\C)$ and $| \mathrm{ R}_G^{\rm irr}(M) |<\infty$, then $c_G=24 $ and the sum in \eqref{sum} is zero. 
\end{conj} 

This paper mainly examines supporting evidence of Conjecture \ref{conj1} as follows: In Section \ref{Sec3}, we show (Proposition \ref{key3}) that Conjecture \ref{conj1} is correct, if $\mathrm{ R}_G^{\rm irr }(M) $ satisfies a Galois descent of $K_3$ shown in \cite{Zic}, which implies that the sum \eqref{sum} factors through $K_3^{\rm ind} (\Q) \cong \Z/24$. As a corollary, we show many examples of hyperbolic 3-manifolds satisfying the conjecture (Theorem \ref{lld4024}). We also examine relations between the conjecture and the dilogarithm function and approximately give supporting evidence of the conjecture (Corollary \ref{ll26g4}). Section \ref{Sec1} proves the conjecture for some non-hyperbolic 3-manifolds, and Section \ref{HG2} proposes another conjecture for hyperbolic 3-manifolds with a torus boundary and discusses further problems.

\subsection{Acknowledgments} 

The author thanks Dongmin Gang, Yuichi Kabaya, Masanori Morishita, Jun Murakami, and Yuji Terashima for their useful comments. He also sincerely expresses his gratitude to Christian Krogager Zickert for suggesting an idea for shortly proving Propositions \ref{key3} and \ref{l33l244}, and for invaluable comments.

\section{Galois action on the sum \eqref{sum}, and dilogarithm identities} 

\label{Sec3} 

In this section, we show the key propositions and give supporting evidence of Conjecture \ref{conj1} from the viewpoint of the Galois descent above.

To begin, we should notice that the imaginary part of Problem \ref{conj12} is almost clear: 

\begin{prop}\label{ll24} The imaginary part of the sum \eqref{sum} is zero. 
\end{prop} 

\begin{proof} For a homomorphism $\rho: \pi_1(M) \ra G$, we denote by $\bar{\rho}$ the conjugate representation. We can choose representatives $ \rho_1, \dots , \rho_m ,\overline{\rho_1}, \dots , \overline{\rho_m}, \rho'_1 ,\dots, \rho_n' $ of $ \pi_0( \mathrm{ R}^{\rm irr}_{G }(M) ) $ such that $ [\rho'_i]= [\overline{\rho'_i}] \in \pi_0( \mathrm{ R}^{\rm irr}_{G }(M) ) . $ If $A_{\rho}$ is a flat connection associated with $\rho$, the conjugate $\overline{A_{\rho}}$ equals $A_{\overline{\rho}}$. Thus, $\mathrm{Im}(\CS(\rho)) = - \mathrm{Im}(\CS(\overline{\rho})) $ by definition. In particular, $\mathrm{Im}(\CS(\rho_i)) = - \mathrm{Im}(\CS(\overline{\rho_i})) $ and $\mathrm{Im}(\CS(\rho_j')) =0$. Hence, the imaginary part of the sum \eqref{sum} is zero as required. 
\end{proof} 

In conclusion, we shall focus on only the real part of the Chern-Simons invariants. For simplicity, this paper focuses on the case $G=\SL_2(\mathbb{C})$. From the viewpoint of Galois groups, let us give a brief proof of Conjecture \ref{conj1} in some situations. For a subfield, $F \subset \C$, let $\mathrm{ R}_{\SL_2(F)}^{\rm irr}(M) $ be the set of the conjugacy classes of all irreducible representations $\pi_1(M) \ra \SL_2(F) .$ 

\begin{prop}\label{key3}Let $G=\SL_2(\mathbb{C})$. 
Let $F/\Q$ be a Galois extension of finite degree with an embedding $F \hookrightarrow \C$. Assume that the inclusion $\mathrm{ R}_{\SL_2(F)}^{\rm irr}(M) \subset \mathrm{ R}_{\SL_2(\C)}^{\rm irr}(M) $ is bijective as a finite set, and is closed under the action of the Galois group $\mathrm{Gal}(F/\Q)$. Then, Conjecture \ref{conj1} is true. 
\end{prop} 

\begin{proof}  This short proof is essentially due to Zickert \cite{Zic2} in a private communication. Zickert \cite{Zic} defines ``the extended Bloch group $\widehat{\mathcal{B}}(F)$" such that the Chern-Simons invariant can be described as the composite map $\mathrm{ R}_{\SL_2(F)}^{\rm irr}(M) \ra \widehat{\mathcal{B}}(F) \stackrel{\widehat{L}}{\lra} \C/\Z $. Here, it is shown that $\widehat{L}$ is a homomorphism, and the former map is equivariant under the action of $\mathrm{Gal}(F/\Q)$, and that $\widehat{\mathcal{B}}(F) $ is isomorphic to the indecomposable part of the algebraic $K$-group $K_3^{\rm ind}(F)$; see Section 5 and Theorem 7.13 in \cite{Zic}. By assumption, the sum \eqref{sum} lies in the invariant subgroup $ \widehat{\mathcal{B}}(F)^{ \mathrm{Gal}(F/\Q)} $. Since $ \widehat{\mathcal{B}}(F)^{ \mathrm{Gal}(F/\Q)} = \widehat{\mathcal{B}}(\Q)$ is a Galois descent (\cite[Corollary 7.15]{Zic}) and $K_3^{\rm ind}(\Q) \cong \Z/24 $ is well-known to be true, the sum \eqref{sum} with $c_G=\Z/24$ vanishes.  
\end{proof}

Now let us examine some examples. For a knot $K$ in the 3-sphere and $r=p/q \in \Q $, let $ M_{p/q} (K)$ be the closed 3-manifold obtained by the $(p/q)$-surgery along $K$. 

\begin{thm}\label{lld4024} 
Let $K$ be a twist knot. Suppose that $p$ is even and relatively prime to $q$. Then, Conjecture \ref{conj1} holds in the case where $M$ is $ M_{p/q} (K)$. 
\end{thm} 

\begin{proof}  
We may check that $ M_{p/q} (K)$ satisfies the condition in Proposition \ref{key3}. Let $n \in \Z$ with $n\neq 0$ such that $K$ is the twist knot having $|n|$ right-handed half twists (left-handed, if $n$ is negative). Then, the fundamental group of $S^3 \setminus K$ has a presentation, 
\[ \pi_1( S^3 \setminus K) = \langle a, b \mid w^n a = bw^n \rangle , \ \ \ \textrm{where } w = ba^{-1} b^{-1}a. \]  
From \cite{Yoon}, let us review explicitly the set $\mathrm{ R}_{\SL_2(\C)}^{\rm irr}(M)$. Take the Chebyshev polynomials $S_k(z) \in \Z[z]$ defined by $S_{k+1}(z) = z S_k(z)- S_{k-1}(z)$ for all $k \in \Z$ inductively, where the initial conditions are defined by $S_0(z) = 0, S_1(z) = 1.$ As in \cite{Yoon}, we write $S_k(z)$ simply as $S_k$ for short. From Remarks 6 and 7 in \cite{Yoon}, take the Laurent polynomials of the form, 
\[ F(m, z) = S_n(S_n - S_{n-1})(m^2 + m^{-2}) -z S_n(S_n- S_{n-1}) + 1, \] 
\[ E_{\gamma} (m,z) = -(z - 2)S^2_n m^4 - (z - 2)(S_n - S_{n-1} )m^2 + (S_n- S_{n-1} )^2 \in \Z[m^{\pm 1},z^{\pm 1 }] . \] 
Then, as shown in \S 2.2 in \cite{Yoon}, the set $\mathrm{ R}_{\SL_2( \C)}^{\rm irr}(M)$ is in a 1:1-correspondence with  
\begin{equation}\label{pp44} \{ (m,z) \in \mathbb{C}^{\times} \times \C \mid F(m,z)=0, m^{p} E_{\gamma} (m,z)^q =1\}_{ (m,z) \sim (m^{-1},z)} .\end{equation} 
Here, for $(m,z)$ in this set, the associated representation $\rho : \pi_1(M) \ra SL_2(\C)$ is given by 
\[ \rho(a) = \begin{pmatrix} 
m& 1 \\ 
0 & m^{-1} \\ 
\end{pmatrix} , \ \ \ \rho(b)= \begin{pmatrix} 
m & 0 \\ 
-z& m^{-1} \\ 
\end{pmatrix} \] 
Since $p $ is even, the equation $F(m,z)=0$ implies that $m^2 = e \pm \sqrt{ g}$ and the equation $ m^{p} E_{\gamma} (m,z)^q =1$ reduces to $ h m^2 -k =0$ for some rational functions $e,g,h,k \in \Q(z)$. Thus, $ h m^2 -k = h ( e \pm \sqrt{ g}) -k = 0$ is equivalent to $\pm \sqrt{ g} = k/h -e$: that is, $ g= ( k/h -e )^2 $. Let $f (z)\in \Q[z^{\pm 1}]$ be a numerator of the common denominator of $ g- ( k/h -e )^2. $ Let $E/\Q$ be the field extension by $f(z),\sqrt{ g}$, and $F/E$ be that by $(e \pm \sqrt{ g})^{1/2}$. Then, by construction, the number field $F$ satisfies the condition in Proposition \ref{key3}, as desired.  
\end{proof}

As an application of Theorem \ref{lld4024}, we can point to a non-trivial identity of the dilogarithm function (Corollary \ref{ll26g4} below). Let $K$ be the figure eight knot, for simplicity. From \cite{HM}, we review the resulting computation of the Chern-Simons invariants from the viewpoint of saddle point methods. Consider the classical dilogarithm $\Li_2(z)= - \int_{0}^z \frac{\Log(1-t)}{t} {\rm d}t$. The codomain is $\C/\pi^2 \Z$. Define the (potential) function $\tilde{V}$ on the area $\{ (z,w) \in \C^2 \mid zw \neq 0 \}$ by 
\[ \tilde{V}(z,w):= - \mathrm{Li}_2(zw) + \mathrm{Li}_2(z/w)+ \frac{p}{4}( \Log z)^2 - \Log z \Log w \in \C/\pi^2 \Z ,\] 
and the following equation in \cite[(4.2)]{HM}: 
\begin{equation}\label{sum29} \left\{ 
\begin{array}{l} 
w= \frac{z+z^{p/2}}{z^{p/2} z+1 } ,\\ 
z^2 - ( \frac{z+z^{p/2}}{z^{p/2} z+1 } +1 +\frac{1+z^{p/2}z }{z^{p/2} +z } )z +1 =0. 
\end{array} 
\right. \end{equation} 
If $p$ is even and $|p|\geq 6$, it is shown that there is a bijection, 
\[\mathcal{B}: R_G^{\rm irr}(M_{p/1}(K) ) \stackrel{1:1}{\longleftrightarrow} \{ (z,w) \in \C^2 \mid \eqref{sum29}\textrm{ is satisfied, and }\mathrm{Im}(z) \geq 0 \} . \] 
Let $(\zeta_1, \omega_1),\dots, (\zeta_p, \omega_p) $ be the solutions of \eqref{sum29}. For $i \leq p$, choose integers $c_i,d_i \in \Z$ so that 
\[ \frac{\partial \tilde{V}}{ \partial z}(\zeta_i, \omega_i)+ 2 \pi \sqrt{-1}\frac{c_{i}}{\zeta_i}= 
\frac{\partial \tilde{V}}{ \partial w}(\zeta_i, \omega_i)+ 2 \pi \sqrt{-1} \frac{d_i }{\omega_i} =0. \] 
Then, as shown in \cite[Remark 5.7]{HM} (see also \cite{Oht2} for mathematical estimates on the saddle point method), the Chern-Simons invariant can be computed as follows: 
\[ 4 \mathrm{CS}( \mathcal{B}^{-1} (\zeta_i, \omega_i))= \tilde{V}(\zeta_i, \omega_i) + 2\pi \sqrt{-1} ( c_{i} \log \zeta_i + d_{i} \log \omega_i ). \] 
\begin{cor}\label{ll26g4} 
Suppose that $p$ is even and $|p| \geq 6$. Then,  
\[ \frac{6}{\pi^2 } \sum_{i=1}^p \tilde{V}(\zeta_i, \omega_i) + 2\pi \sqrt{-1} ( c_{i} \log \zeta_i + d_{i} \log \omega_i ) = 0 \in \C/\Z.\] 
\end{cor} 

\noindent 

In general, if Conjecture \ref{conj1} is solved for every closed 3-manifold, then the concrete computation of Chern-Simons invariants might produce dilogarithm identities. Furthermore, if the conjugacy class $ \mathrm{ R}^{\rm irr}_G(M) $ as an algebraic variety is of dimension $>0$, we might get new dilogarithm identities with parameters.

In addition, we can numerically check Conjecture \ref{conj1} for other hyperbolic 3-manifolds. In fact, in terms of group (co)-homology and the extended Bloch group \cite{Zic, DZ}, it is not so hard to establish a computer program for computing the Chern-Simons invariants of easy 3-manifolds; see Appendix \ref{HG4}. By using such a program, the author numerically checked the following. 

\begin{prop}\label{ll2433} 
Let $K$ be a hyperbolic knot with crossing number $<8$. Let $p$ satisfy $|p|<7$ and that  $M_{p/1}(K)$ is hyperbolic. Then, the absolute value of the sum \eqref{sum} with $c_G=24$ is bounded by $10^{-6}$. 

\end{prop}

\section{The conjectures of some Seifert 3-manifolds and torus bundles} \label{Sec1} 

While the preceding section concentrates on only hyperbolic manifolds, this section solves Conjecture \ref{conj1} for some non-hyperbolic 3-manifolds. Here, the point is that the set $ \mathrm{ R}^{\rm irr}_G(M)$ might not be closed under any Galois action. Precisely, 

\begin{prop}\label{prop3} 
Let $G=\SL_2(\C),$ and let $M$ be either a Seifert 3-manifold with three singular fibers over $S^2$ or a torus bundle over the circle. Then, the conjecture is true. 
\end{prop}

\begin{proof}
If $M $ is a torus bundle over the circle, then $\mathrm{CS}(\rho) \in \frac{1}{2} \Z $ is known for any $\rho : \pi_1(M) \ra \SL_2(\C)$; see, e.g., \cite[\S 4.2]{CQW}. Thus, the conjecture is obviously true. We may suppose that $ M$ is the Seifert 3-manifold of type $M( 0; (o, 0); (p_1, q_1), (p_2, q_2), (p_3, q_3))$, where $p_i,q_i \in \Z$. Choose integers $s_j$ and $r_j$ satisfying $p_j s_j - q_j r_j =1$.

Let us briefly review the set $ \mathrm{ R}^{\rm irr}_G(M)$ and the Chern-Simons invariants from \cite[\S 3.1]{CQW}. For an integer $n >0$, let $[0 \cdots n ]$ be the set $\{0, 1,\dots, , n\}$, and $[0 \cdots n ]^e$ (resp. $[0 \cdots n ]^o$) is the subset of even (resp. odd) integers in $[0 \cdots n ]$. Then, we can define a bijection $\mathcal{B}$ from the set $ \mathrm{ R}^{\rm irr}_G(M)$ to 
\[ \Bigl\{ (\frac{j_1+ 1}{2},\frac{j_2+ 1}{2},\frac{j_3+ 1}{2},\frac{1}{2}) \in (\frac{1}{2}\Z)^4 \mid j_k \in [ 0 \cdots p_k-2]^{ \epsilon_k} \Bigr\}\] 
\[ 
\cup \Bigl\{ (\frac{j_1+ 1}{2},\frac{j_2+ 1}{2},\frac{j_3+ 1}{2},0 ) \in (\frac{1}{2}\Z)^4 \mid j_k \in [ 0 \cdots p_k-2 ]^{ o} \Bigr\} \] 
where $\epsilon_k = e$ if $q_k$ is odd, and $\epsilon_k = o$ otherwise. It follows from \cite[Proposition 2.1]{CQW} that, for $(n_1, n_2, n_3, \lambda)$ in the set, the fourfold Chern-Simons invariant is computed as $4 \mathrm{CS} ( \mathcal{B}^{-1} (n_1, n_2, n_3, \lambda)) = 4 \sum_{j=1}^3 (r_j n_j^2 /p_j) \in \Q/\Z$.

Thus, the computation of the sum \eqref{sum} can be done in some cases according to parity of $p_k$ and $q_k $. For example, if the $p_k$'s are odd and $q_k $'s are even, the sum can be computed as 
\[ \sum_{\rho \in \mathrm{ R}^{\rm irr}_G(M) } 4\mathrm{CS}(\rho)= 4\sum_{i=1}^{ (p_1-1)/2} \sum_{j=1}^{ (p_2-1)/2}\sum_{k=1}^{ (p_3-1)/2} 2(\frac{ r_1 i^2}{p_1} + \frac{ r_2 j^2}{p_2} +\frac{ r_3 k^2}{p_3} ) = \frac{4}{3}\prod_{j=1}^3 \frac{p_j-1}{2} \sum_{k=1}^3 r_k (\frac{p_k+1}{2}),\] 
which lies in $\frac{1}{3}\Z$. Thus, the sum \eqref{sum} with $c_G=24$ is zero. It is left to the reader to check the other cases. 
\end{proof} 


\section{Reciprocity for 3-manifolds with boundary, and other gauge groups}\label{HG2} 
Now let us discuss the situation where $M$ has a torus boundary. The integral \eqref{sum22} 
ever converges with respect to any homomorphism $\pi_1(M) \ra \SL_2(\C) $; however, given a boundary-parabolic representation $\rho$, we can define the Chern-Simons invariant $\mathrm{CS} \in \C/ \Z $; see \cite[Sections 2--5]{Zic} for the detailed definition. For instance, if $M$ is hyperbolic and $\rho$ is the holonomy representation, the fourfold invariant is known to be equal to the complex volume of $M$. Let $\mathrm{ R}^{\rm para}_{\SL_2(\C)}(M) $ be the set of conjugacy classes of boundary-parabolic representations $\rho : \pi_1(M) \ra \SL_2(\C)$, and $\pi_0(\mathrm{ R}^{\rm para}_{\SL_2(\C)}(M) ) $ be the connected components. The Chern-Simons invariant is locally constant. Thus, similarly to the main conjecture \ref{conj1}, it is sensible to pose the following problem: 

\begin{conj}\label{conj3} 
Let $M$ be an oriented compact 3-manifold with a torus boundary. If $| \mathrm{ R}_{\SL_2(\C)}^{\rm para}(M) |<\infty$, the similar sum $24 \sum_{\rho \in \mathrm{ R}^{\rm para}_{\SL_2(\C) }(M) } \CS (\rho) $ is zero. 
\end{conj} 
By the same discussion as in Proposition \ref{key3}, we can readily show the following:  
\begin{prop}\label{l33l244} 
Take a Galois extension $F/\Q$ of finite degree. Assume a complete set of representatives, $\rho_1, \dots, \rho_m, $ of $\pi_0(\mathrm{ R}^{\rm para}_{\SL_2(\C) }(M) ) $ such that $\mathrm{Im}( \rho_j )\subset \SL_2(F) $ and the set is closed under the action of $\mathrm{Gal}(F/\Q)$. Then, the sum $24 \sum_{[\rho] \in \pi_0(\mathrm{ R}^{\rm para}_{\SL_2(\C) }(M) )} \CS (\rho) $ is zero. 
\end{prop} 
As a corollary, we get supporting evidence:  
\begin{cor}\label{l33l4} Let $K$ be a 2-bridge knot.  
If $M$ is the exterior of $K$, Conjecture \ref{conj3} holds. 
\end{cor} 
\begin{proof} 
The knot group $ \pi_1(S^3 \setminus K) $ has a presentation of the form $\langle x, y \mid W x = yW \rangle $, where $W$ is a word in $x$ and $y$. Riley \cite[Sect. 3]{Riely} defines a polynomial $\Phi_K \in \Z [t^{\pm 1}]$ such that any zero $\zeta $ of $\Phi(K)$ gives rise to a representation $\gamma_{\zeta}: \pi_1(S^3 \setminus K) \ra \SL_2( \C)$ of the form, 
\[ \rho(x) = \begin{pmatrix} 
1& 1 \\ 
0 & 1 \\ 
\end{pmatrix} , \ \ \ \rho(y)= \begin{pmatrix}1& 0 \\ 
\zeta & 1 \\ 
\end{pmatrix} . \] 
Riley shows that any parabolic representation of $K$ is conjugate to such a representation. Let $ F/\Q$ be the field extension by $\Phi_K $. Then, we rearch at the condition of Proposition \ref{l33l244}.  
\end{proof} 
We can get similar results if we replace $K$ with a link of crossing number $<7$. However, Conjecture \ref{conj3} remains mysterious if $\mathrm{ R}^{\rm para}_{\SL_2(\C) }(M) $ is not of finite order.

Finally, we conclude this paper by discussing further problems with other gauge groups. In the introduction, we supposed that $G$ is simply connected and simple over $\C$. In general, through a discussion as in \cite[Section 6]{CS}, if $\pi_1(G)=\Z$ and the homotopy groups $\pi_2(G) = \pi_3(G)=0$ or if $\pi_1(G) = \pi_2(G)\otimes \Q =0$ and $\pi_3(G)\otimes \Q =\Q $, we can define the Chern-Simons invariant whose codomain is $\C/\Z$. Moreover, if $G$ is a Lie group and 3-connected as in $G= \widetilde{\SL_2}(\R)$, the Chern-Simons invariant is defined as an $\R$-valued function. Thus, in the situations on $G$, it might be interesting to investigate a problem similar  to Problem \ref{conj12} even if $R_G^{\rm irr}(M) $ is not of finite order. However, if when $M$ is a rational homology 3-sphere and $G=U(1)$ is taken as the simplest case, the problem becomes almost trivial. In fact, $ R_{G}^{\rm irr}(M) $ is identified with $ \Hom(H_1(M;\Z),U(1)) $, and the Chern-Simons invariant $\mathrm{CS}: \Hom(H_1(M;\Z),U(1)) \ra \Q/\Z$ is known to be a quadratic form; see, e.g., \cite{MOO}. Thus, the sum \eqref{sum} with $c_{U(1)}=48$ must be zero.

\appendix 

\section{Computation of the Chern-Simons invariants of 3-manifolds} \label{HG4} 
Take a closed 3-manifold $M$ with orientation homology 3-class $[M] \in H_3(M;\Z) \cong \Z$ and a homomorphism $\rho: \pi_1(M) \ra \SL_2(\C)$. Following \cite{Nos2}, we establish a procedure to compute the Chern-Simons invariant. As is known (see \cite[Proposition 2.8]{CS}), the invariant is equal to the pairing $ \langle \widehat{C}_2 ,\rho_* [M]\rangle \in \C /\Z$ for some group 3-cocycle $\widehat{C}_2$ of $\SL_2(\C)$. Thus, if $2 \widehat{C}_2$ and the pushforward $ \rho_* [M]$ are concretely described in terms of group cohomology, the doubled Chern-Simons invariant can be computed with a help of a computer.

We start by reviewing the group homology of $G=\SL_2(\C)$. Let $Y$ be a set acted on by $G $, and $C_n (Y)$ be the free $\Z$-module generated by $(n+1)$-tuples of $Y$. Consider the boundary map $\partial_n : C_n (Y) \ra C_{n-1} (Y)$ by 
\[ \partial_n(y_0,y_1,\dots, y_n )= \sum_{i=0}^n (-1)^i (y_0,\dots,y_{i-1}, y_{i+1},\dots, y_n ). \] 
If $Y =\SL_2(\C) $ with canonical action, the homology of $ C_*(Y) \otimes_{\Z[ \SL_2(\C) ]} \Z $ is called {\it the group homology of $\SL_2(\C)$}. Since $\SL_2(\C)$ acts on the complex projective space $\mathbb{C}P^1 $ transitively, we have a surjection $h: G\ra \mathbb{C}P^1$ as a $G$-equivariant map, which is called a Hopf map. Consider the subcomplex of the form, 
\[ C_n^{h \neq } (\SL_2(\C) ) := \Z \langle (y_0,y_1,\dots, y_n ) \in (\SL_2(\C) )^{n+1} \mid h(y_i) \neq h(y_j) \ \ (\textrm{if } i \neq j) \rangle \] 
Then, as is known (see \cite[\S 3]{DZ}), the homology of $ C_n^{h \neq } (\SL_2(\C) ) \otimes_{\Z[ \SL_2(\C) ]} \Z$ is isomorphic to the group homology of $\SL_2(\C)$, and any 3-cycle in $ C_*(Y) \otimes_{\Z[ \SL_2(\C) ]} \Z $ with $Y =\SL_2(\C)$ is homologous to a 3-cycle in the subcomplex $ C_n^{h \neq } (\SL_2(\C) ) \otimes_{\Z[ \SL_2(\C) ]} \Z $. In the paper \cite[Theorem 4.1]{DZ}, the doubled of the Chern-Simons 3-cocycle, $2 \widehat{C}_2$, is explicitly described as a map $ \frac{ -1}{2\pi^2}\widehat{L} \circ \widehat{\lambda }: C_3^{h \neq } (\SL_2(\C) ) \ra \C/\Z$ in terms of the extended Bloch group and dilogarithms; see \cite[Section 3]{DZ} for details.

Next, we will describe the pushforward $ \rho_* [M]$. We choose a genus $g$ Heegaard decomposition of $M. $ Then, $M$ consists of a single 0-handle, $g$ pieces of 1-, 2-handles, and a single 3-handle. By the van Kampen theorem, we have a group presentation $\langle x_1,\dots, x_g \mid r_1, \dots, r_g \rangle$ of $\pi_1(M)$. Moreover, the (reduced) cellular complex of the universal covering space of $M$ can be described as 
\begin{equation}\label{z0} C_*(\widetilde{M} ;\Z): 0 \ra \Z[\pi_1(M)] \stackrel{\partial_3}{\lra} \Z[\pi_1(M)]^g \stackrel{\partial_2}{\lra}\Z[\pi_1(M)]^g \stackrel{\partial_1}{\lra}\Z[\pi_1(M)] \ra \Z \ra 0,\end{equation} 
where $\Z[\pi_1(M)]$ is the group ring of $\pi_1(M)$. As is known, the boundary maps $\partial_2$ (resp. $\partial_3)$ can be described in terms of Fox derivative (resp. the ``taut identity"). For example, if $M$ is the closed 3-manifold obtained as the $(p/1)$-surgery along a knot in $S^3$, i.e., $M=M_{p/1}(K)$, we can explicitly describe $\partial_2$ and $\partial_3$; see \cite[Corollary 3.5]{Nos2} for details.

Following \cite[\S 5]{Nos2}, the $\pi_1(M)$-equivariant chain map $\rho_*: C_*(\widetilde{M} ;\Z) \ra C_*(G)$ is constructed as follows. Let $\{a_j\}_{j=1}^g$ and $\{ b_j\}_{j=1}^g$ be a canonical basis of $ C_1(\widetilde{M} ;\Z)$ and $C_2(\widetilde{M} ;\Z)$, respectively. Let $c_0$ be the identity. For $A \in \pi_1(M)$, define $c_1(A a_i) := (\rho(A), \rho(A x_i)) $. If $r_i$ is expanded as $ x_{i_1}^{\epsilon_1} x_{i_2}^{\epsilon_2} \cdots x_{i_n}^{\epsilon_n} $ for some $ \epsilon_k \in \{ \pm 1 \}$, we define 
$$c_2( A b_i) = \sum_{m: 1 \leq m \leq n}\epsilon_m ( \rho(A ) ,\rho(A x_{i_1}^{\epsilon_1} x_{i_2}^{\epsilon_2} \cdots 
x_{i_{m-1}}^{\epsilon_{m-1}} x_{i_m}^{(\epsilon_m -1)/2} ),\rho( A x_{i_1}^{\epsilon_1} x_{i_2}^{\epsilon_2} \cdots x_{i_{m-1}}^{\epsilon_{m-1}} 
x_{i_m}^{(\epsilon_m +1)/2})) \in C_2(G ). $$ 

Then, the commutativity $ \partial_1^{\Delta} \circ c_1 = c_0 \circ \partial_1 $ and $ \partial_2^{\Delta} \circ c_2 = c_1 \circ \partial_2 $ is true. Let $\mathcal{O}_M $ be the canonical basis of $ C_3(\widetilde{M} ;\Z)$. Notice that $\partial^{\Delta}_2 \circ c_2 \circ \partial_3 (\mathcal{O}_M )= c_1 \circ \partial_2 \circ\partial_3 (\mathcal{O}_M ) =0$, that is, $c_2 \circ \partial_3 (\mathcal{O}_M ) $ is a 2-cycle. If we expand $c_2 \circ \partial_3 (\mathcal{O}_M ) $ as $\sum n_i (g_0^i, g_1^i, g_2^i)$ for some $ n_i \in \Z, g_j^i \in G $, then with a choice of $v_0 \in G$ such that $v_0 $ is different from all $ g_k^\ell$, $\mathcal{O}_M ':=- \sum n_i (v_0, g_0^i, g_1^i, g_2^i)$ satisfies $\partial_3^{\Delta }(\mathcal{O}_M ') = c_2 \circ \partial_3 (\mathcal{O}_M ) $. Therefore, the correspondence $\mathcal{O}_M \mapsto \mathcal{O}_M ' $ gives rise to a chain map $c_3 : C_*(\widetilde{M} ) \ra C_*(G ) $, as desired. Here, $c_3$ up to homotopy is independent of the choice of $v_0.$ In conclusion, the class $ \mathcal{O}_M ' \otimes 1 $ in $ C_3(G) \otimes \Z $ is a representative of the pushforward $ \rho_* [M]$.

In our experience, if $M$ is hyperbolic, fortunately, the 3-class $ \mathcal{O}_M ' $ often lies in the subcomplex $C_n^{h \neq } (\SL_2(\C) ) $. Thus, it is not so hard to compute numerically the pairing $ \langle 2 \widehat{C}_2,\rho_* [M]\rangle \in \C /\Z$. As stated in Proposition \ref{ll2433}, for instance, if $M=M_{p/1}(K)$ and $|p| <8$ and $\mathrm{cr}(K)<8$, we can easily compute the pairings numerically.

\vskip 1pc

\normalsize

\noindent 

Department of Mathematics, Tokyo Institute of Technology 2-12-1 Ookayama, Meguro-ku Tokyo 152-8551 Japan

E-mail address: {\tt nosaka@math.titech.ac.jp}

\end{document}